\documentclass[12pt]{amsart}

\usepackage{amssymb, amsmath}
\usepackage{amsfonts, a4}
\usepackage{amsthm}
\usepackage{eufrak}
\usepackage{epsfig}
\usepackage{subfig}

\usepackage[all]{xy}

\theoremstyle{plain}

\newcommand{\RR}{{\bf R}}

\newcommand{\ZZ}{{\bf Z}}

\numberwithin{equation}{section}

\author{Fran\c cois Maucourant}
\address{Universit\'e Rennes I, IRMAR, Campus de Beaulieu 35042 Rennes
cedex -  France}
\email{francois.maucourant@univ-rennes1.fr}

\author{Barak Weiss}

\address{Ben Gurion University, Be'er Sheva, Israel 84105}
\email{barakw@math.bgu.ac.il}

\newif\ifdraft\drafttrue

\draftfalse

\font\sb = cmbx8 scaled \magstep0

\font\sn = cmssi8 scaled \magstep0
\font\si = cmti8 scaled \magstep0

\long\def\combarak#1{\ifdraft{\si #1 }\else\ignorespaces\fi}

\long\def\comfm#1{\ifdraft{\sb #1 }\else\ignorespaces\fi}

\newcommand\name[1]{\label{#1}{\ifdraft{\sn [#1]}\else\ignorespaces\fi}}

\newcommand\eq[2]{{\ifdraft{\ \tt [#1]}\else\ignorespaces\fi}\begin{equation}\label{eq:
#1}{#2}\end{equation}}

\newcommand {\equ}[1]     {\eqref{eq: #1}}

\newcommand{\R}{{\mathbb{R}}}
\newcommand{\M}{{\operatorname{M}}}

\newcommand{\Z}{{\mathbb{Z}}}

\newcommand{\PP}{\mathcal{P}}

\newcommand{\vv}{{\mathbf{v}}}
\newcommand{\vu}{{\mathbf{u}}}
\newcommand{\vw}{{\mathbf{w}}}

\newcommand{\SL}{\operatorname{SL}}
\newcommand{\PSL}{\operatorname{PSL}}

\newcommand{\gmg}{\Gamma \backslash G}

\newcommand{\dist}{{\mathrm{dist}}}

\newcommand{\til}{\widetilde}

\newcommand{\supp}{{\rm supp}}

\newcommand{\sm}{\smallsetminus}

\newcommand{\vre}{\varepsilon}

\font\sb = cmbx8 scaled \magstep0

\newcommand {\ignore}[1]  {}

\newtheorem{thm}{Theorem}[section]

\newtheorem{lem}[thm]{Lemma}

\newtheorem{cor}[thm]{Corollary}

\newtheorem{remark}[thm]{Remark}

\begin{document}
\title[Lattice action on the plane]{Lattice actions on the plane revisited}


\begin{abstract} 
We study the action of a lattice $\Gamma$ in the group $G=\SL(2, \RR)$
on the plane. We obtain a
formula which simultaneously describes visits of an orbit $\Gamma \vu$
to either a fixed ball, or an expanding or contracting family of
annuli. We also discuss the `shrinking target problem'. Our results
are valid for an explicitly described set of
initial points: all $\vu \in \RR^2$ in the case of a cocompact lattice, and
all $\vu$ satisfying certain diophantine conditions in case $\Gamma =
\SL(2,\Z)$. The proofs combine the method of Ledrappier with effective
equidistribution results for the horocycle flow on $\gmg$ due to
Burger, Str\"ombergsson, Forni and Flaminio. 
\end{abstract}

\maketitle

 \section{Introduction, statement of the results}
A classical problem in ergodic theory is to understand the
distribution of orbits for the action of a group on a
space. This has been particularly well-studied under the hypotheses
that the
acting group $\Gamma$ is {\em amenable} and preserves a {\em finite
measure}. Removing these two assumptions leads to a realm which is not
sufficiently understood. Our purpose in this note is to describe some
features which arise when one studies a non-amenable group acting on a
space preserving an infinite Radon measure. We will consider the simplest
setup with these features. Namely, let $\Gamma$ be a 
lattice in $G=\SL(2,\RR)$, that is, a discrete subgroup of finite
covolume. It acts on the punctured plane $\PP= \RR^2 \sm \{0\}$ by
linear transformations, preserving  
Lebesgue measure. It is well-known that this action is
ergodic. Moreover, when $\Gamma$ is cocompact all orbits are dense,
and when $\Gamma$ is non-uniform, any 
orbit is either discrete or dense. 

Consider an orbit $\Gamma \vu$ and an increasing family
$\{\Gamma_T: T>0\}$ of finite sets in $\Gamma$. We will refer to 
$\Gamma_T \vu \subset \PP$ as a `cloud'; we wish to understand its
distribution for large values of $T$. For example one can
ask for the frequency of visits to a fixed 
ball in the plane. One can also consider
the behavior of the orbit under rescaling, i.e. the frequency of
visits to a family of balls; if the balls are expanding with $T$ this
corresponds to the `large scale' behavior of the orbit and if the
balls are shrinking this corresponds to the behavior of the orbit `at
a point.' The answers to these questions turn out to depend rather
delicately on the choice of the averaging sets $\Gamma_T$ and the
initial point $\vu$. 

Fix a norm $\| \cdot\|$ on $\M_2(\RR)$, the space of two by two
matrices with real entries.
  Define for any $T>0$ the set 
$$\Gamma_T=\left\{ \gamma \in \Gamma \, : \, \| \gamma \| \leq T
\right\},$$ 
let $f$ be a compactly supported function on $\PP$ and $\vu
\in \RR^2$. 

\begin{figure}[h]
\begin{center}

\epsfig{file=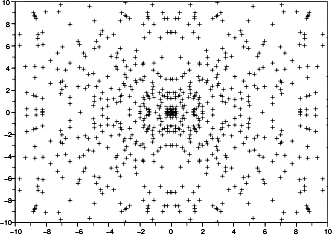,width=300.00pt,height=300.00pt}
\caption{`Cloud' for the cocompact lattice $\SL_1(D_{2,3}(\ZZ)),\, \vu = (1,0), \, T=100.$}

\end{center}
\end{figure}

\ignore{
\comfm{ The preceding picture illustrates the case where $\|\cdot \|$ is the usual $L^2$ norm, 
the cocompact lattice is the subgroup usually denoted by $SL(1,D_{2,3}({\bf{Z}}))$, more precisely:
$$\Gamma=\left\{ \left[ \begin{array}{cc} x+\sqrt{2}y & z+\sqrt{2}t \\
3(z-\sqrt{2}t) & x-\sqrt{2}y  \end{array}\right] \; : \; (x,y,z,t) \in
{\bf{Z}}^4, \; x^2-2y^2-3z^2+6t^2=1 \right\},$$  
with the initial vector $\vu=\left( \begin{array}{c} 1\\ 0 \end{array} \right)$ and time $T=80$.
}}
The asymptotics of the orbit-sum 
$$S_{f,\vu}(T)=\sum_{\gamma \in \Gamma_T} f(\gamma \vu)$$
were studied in \cite{l, no, gw}. 
Write $\vu=\left[ \begin{array}{c} u_1 \\ u_2
\end{array}\right] \in \RR^2$, 
 and define a norm $| \cdot |$ and a `product' $\star$ on $\RR^2$ by
  $$|\vv|=\max \left\{|v_1|, |v_2| 
\right\}, \ \ \ 
\vv\star \vu=\Biggl\| \left[
     \begin{array}{cc}
      -u_2 v_1 & u_1v_1 \\
      -u_2 v_2 & u_1v_2 \\
     \end{array}
    \right]\Biggr\|,$$
where $\vv=\left[ \begin{array}{c} v_1\\  v_2 \end{array}\right]$.
Let $dx$ denote the
Lebesgue measure on $\PP$. It was shown in the above-mentioned
papers (see particularly \cite[\S12.4]{gw}) that
$$
\frac{S_{f,\vu}(T)}{T}
\ \longrightarrow_{T \to \infty} \  \frac{2}{\mu(\gmg)} \int_{\PP}
\frac{f(\vv)}{\vv \star \vu}\, d\vv .
$$

We would like to understand $S_{f,\vu}(T)$ 
at a {\it finite} time $T$, i.e. obtain an effective error estimate in
this asymptotic
formula. In particular we would
like to be able to change the function $f$ and the initial point $\vu$
depending on the time $T$. Before stating our results we introduce
some notation. 

Write $\supp \, f=\overline{f^{-1}(\RR \sm \{0\})}$, and set 
\[\begin{split}
r(f) & = \inf_{\vv \in \supp \, f } |\vv|, \ \ 
\ \ \ \ \  
R(f)=\sup_{\vv \in \supp \, f } 
|\vv|, \ \ \ \ \  v(f)=
\frac{R(f)}{r(f)}.
\end{split}
\]
The homogeneous space $\gmg$ carries a finite measure $\mu$ invariant
under the right action of $G$; we will normalize this measure by
assuming that its lift to $G$ satisfies \equ{eq: haar normalization}. 
Fix $\theta \in (0,1]$. We say that a continuous compactly supported function $f$ on $\PP$ is
{\em $\theta$-H\"older} if $\|f\|_\theta < \infty$, where 
\eq{eq: Holder norm}{
\|f\|_\theta=\sup |f|+ \left( \int_\PP f(\vv)^2 \frac{d\vv}{|\vv|^2}
\right)^{1/2}+\sup_{0<|x-y|\leq |x|/2} \frac{ |x|^\theta
|f(x)-f(y)|}{|x-y|^\theta}.
}

\begin{thm}\name{thm: main, cocompact}
For a cocompact lattice $\Gamma$ in $G$ there are positive constants
 $c$ and $\delta_0$
such that for any $\theta \in (0,1]$, there is a
positive constant $C_\theta$ such that the following holds. For any $\vu \in \PP$
and for any $\theta$-H\"older $f: \PP \to \RR$, of
compact support, let
\eq{eq: defn D0}{
D_0=D_0(\vu,f)=\max \left( \frac{R(f)}{|\vu|}, \frac{|\vu|}{r(f)} \right)
}
and
\eq{eq: defn B}{
B=B(\vu,f)= \left( \frac{R(f)}{|\vu|}\right)^{-\theta \delta_0}(\log v(f)+1).
}
Then for any 
\eq{eq: choice of T0}{
T>T_0=c D_0 
}
one has
 \eq{eq: main}{
\left|S_{f,\vu}(T)- \frac{2T}{\mu(\gmg)} \int_{\PP}
\frac{f(\vv)}{\vv \star \vu}\, d\vv \right| \leq C_\theta 
\|f\|_{\theta} \frac{R(f)}{|\vu|} \left( D_0+B T^{1-\theta \delta_0}\right).
}
\end{thm}

\begin{remark}
\begin{enumerate}
\item
Our proof shows one may take $\delta_0 = \delta_{\Gamma}/21$, where
$\delta_{\Gamma} \leq 1/2$ satisfies $\delta_\Gamma(\delta_\Gamma -1)
\leq \lambda,$ with $\lambda < 0$ the first eigenvalue of the Laplacian on
$\gmg$. Our exponent is not optimal. 
\item
 The inequality \equ{eq: main} behaves as should be expected under
rescaling. More precisely, for any $\lambda>0$, if one replaces $f(\vv)$ by
$g(\vv)=f(\lambda \vv)$, and $\vu$ by $\vw=\lambda^{-1} \vu$,
then $D_0(\vu,f)=D_0(\vw,g)$, $B(\vu,f)=B(\vw,g)$, $T_0$ does not
change either, and both sides of \equ{eq: main} are unaffected. 
\end{enumerate}
\end{remark}

The explicit error term in Theorem \ref{thm: main, cocompact} is useful for
studying the asymptotic behavior of an orbit under
rescaling. For an `expansion coefficient' $\rho>0$ consider the function 
$\displaystyle{f_\rho(x) = f\left(\frac x{\rho} \right)}.$
Then for a parameter $\alpha$,  $f_{T^\alpha}$ describes a one-parameter
family of 
functions, and we are interested in sampling them with the cloud
$\Gamma_T \vu$. For instance, if $f$ is the indicator function of an
annulus of radius 1 then the orbit-sum $S_{f_{T^\alpha}, \vu}(T)$ describes the
number of orbit points in the cloud contained in the similar annulus
of radius $T^{\alpha}$. Since the diameter of the cloud is
approximately $T$, if 
$\alpha >1$ the orbit-sum will vanish for large $T$, and similarly for
$\alpha<-1$. However as long as the expansion of the cloud is faster
than that of the support of the expanded function, the cloud
equidistributes in the support of the function, with respect to the
same asymptotic density as in Theorem \ref{thm: main,
cocompact}. Namely we have:

\begin{cor}\name{cor: applications}
Given a cocompact lattice $\Gamma$ in $G$ and $-1 < \alpha <1$, $0<
\theta\leq 1$, there is $\delta>0$ such 
that for any $\vu \in \PP$ and any compactly supported
$\theta$-H\"older function $f$ on $\PP$ there are positive $T_0$ and $C$ such that
for all $T>T_0$, 
\eq{main cor}{
\left |\frac1{T^{1+\alpha}} \sum_{\gamma \in \Gamma_T}
 f\left(\frac{\gamma \vu}{T^{\alpha}} \right) - \frac{2}{\mu(\gmg)} \int_{\PP}
\frac{f(\vv)}{\vv \star \vu}\, d\vv \right | < CT^{-\delta} .}
\end{cor}

\begin{remark}\name{remark 1}
\begin{enumerate}
\item
Adapting the arguments used in the proof of Corollary \ref{cor:
applications} one can show that for any continuous compactly supported
function $\phi$ on $\PP$ for which $\displaystyle{\frac{\phi(T)}{T}
\to 0}$ and $T \phi(T) \to \infty$, one has
$$
 \ \frac{S_{f_{\phi(T)}, \vu}}{T\phi(T)} \to \frac{2}{\mu\left(\gmg\right)}\int_{\PP}
\frac{f(\vv)}{\vv\star \vu} \, d\vv.$$
\item
The case $\alpha =1$, that is the asymptotic behavior of $\frac1{T^2}
S_{f_T, \vu}$ was studied
in \cite{M}. In this case the asymptotic density is
different. 
\end{enumerate}
\end{remark}

\begin{figure}[h]
\subfloat[$T=100$]{
\epsfig{file=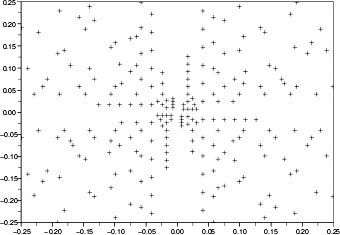,width=140.00pt,height=140.00pt}
}
\subfloat[$T=25$]{
\epsfig{file=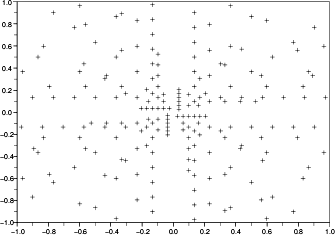,width=140.00pt,height=140.00pt}
}
\subfloat[$T=9$]{
\epsfig{file=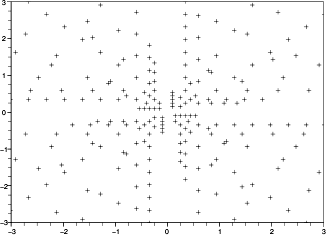,width=140.00pt,height=140.00pt}
}
\caption{The `scaling property'. For $\Gamma = \SL(2,\ZZ)$, parts of the
orbit of three different points are shown, at different scales and
with different values of $T$. Each contains approximately 200 points.}

\end{figure}
 
\ignore{
first three are drawn for $\Gamma=SL(2,{\bf Z})$, $\vu$ unit vector of
slope $\sqrt{2},\sqrt{3}$ and $e$ respectively, $\| \cdot \|$ the
$L^2$ norm, the last one for the cocompact lattice
$SL(1,D_{2,3}({\bf{Z}}))$. The first one is for $T_1=100$ at scale
$S_1=0.25$, the second one is for $T_2=25$ at scale $S_2=1$, and the
third for $T_3=9$ and $S_3=3$. The last one is for $T_4=100$, scale
$S_4=3$, vector $\vu=(1,0)$. They are very similar, as expected
because $T_1S_1=T_2S_2\simeq T_3S_3$; they contains approximately 200
points each. 
}

Theorem \ref{thm: main, cocompact} does not hold in the non-uniform
case. For example, there are discrete orbits for $\Gamma$ in the
plane, and these certainly  will not satisfy \equ{eq: main} if $\supp f$
does not intersect the orbit; consequently, for a
fixed $T$ the conclusion of Corollary \ref{cor: applications} will also fail for all
$\vu$ sufficiently close to a point with a discrete orbit. The
behavior of the orbit will depend in a subtle way on the diophantine
properties of the slope of the initial vector; to make this precise,
we will need a bit of notation. 

Let $z \in [0,1)$, and denote $z=[0;a_1,a_2,\ldots]$ its continued
fraction expansion, $p_k,q_k$ its convergents. Let $t_k=- \log \left|
z- \frac{p_k}{q_k} \right|$; the theory of continued fractions (see
\cite{HW}) tells us that $(t_k)_k$ is an increasing sequence, and that 
 \eq{eq: tk bounds}{
 a_{k+1}q_k^2  \leq e^{t_k} \leq (a_{k+1}+2)q_k^2. 
 }
 Define, in the case where $z$ is irrational
 \[ 
  \hat{\xi}(z,\tau_1,\tau_2) = \max \left( a_k \, : \, \tau_1
\leq t_{k+1}, \,  t_{k-1}\leq \tau_2 \right)  
 \]
(if the set on the right hand side is empty we set $\hat{\xi}(z,
\tau_1, \tau_2) = e^{\tau_2}$).
 If $z$ is rational, the sequences
$(a_k)$, $(q_k)$ and $(t_k)$ are finite. Let $k_0$ be the
length, that is $z=p_{k_0}/q_{k_0}$. If $\tau_2\leq t_{k_0-1}$, $
\hat{\xi}(z,\tau_1,\tau_2)$ is defined by the preceding formula; if
not,  
 \[ 
  \hat{\xi}(z,\tau_1,\tau_2) = \min \left(e^{\tau_2}, \max \left( \max \{a_k \, :
\, \tau_1 \leq t_{k+1} \} , e^{\tau_2}/q_{k_0}^2\right)
\right) . 
\]

 If $\vu= \left[ \begin{array}{c} \vu_x \\ \vu_y \end{array} \right]
\in \PP$, denote by $z$ the unique real number in the set $[0,1] \cap
\{ \vu_x/\vu_y, \vu_x/\vu_y +1, -\vu_y/\vu_x, -\vu_y/\vu_x+1 \}$. We
define $\hat{\xi}(\vu,\tau_1,\tau_2)=\hat{\xi}(z,\tau_1,\tau_2)$.  

\begin{thm}
\name{thm: main, SL(2,Z)}
For $\Gamma=\SL(2,\ZZ)$ there are positive constants $c$ and
$\delta_0$ such that for any $\theta \in (0,1]$, there exists
$C>0$ such that the following holds. For any $\vu \in \PP$ 
and for any compactly supported $\theta$-H\"older map $f: \PP \to
\RR$, let $D_0$ and $B$ be as in \equ{eq: defn
D0} and \equ{eq: defn B}, and let 
$$
\hat{\xi}_{f,T,\vu}= \hat{\xi}\left( \vu,\log \left( \frac{T
|\vu|}{R(f)} \right), \log \left(\frac{T |\vu|}{r(f)}\right) \right). 
$$
Then for any $T>c D_0$ such that
\eq{eq: choice of T1}{
T \geq c \ \frac{| \vu | \hat{\xi}_{f,T,\vu} }{R(f) }
}
one has
 \eq{eq: mainfv2}{
\left|S_{f,\vu}(T)- \frac{2T}{\mu(\gmg)} \int_{\PP}
\frac{f(\vv)}{\vv \star \vu}\, d\vv \right| \leq C
\|f\|_{\theta} \frac{R(f)}{|\vu|} \left( D_0+B T^{1-\theta \delta_0}
\hat{\xi}_{f,T,\vu}^{\theta \delta_0} \right). 
}
\end{thm}

\begin{remark}
\begin{enumerate}\item
Examining our argument one sees it is possible to take $\delta_0 =
1/48$ in Theorem \ref{thm: main, SL(2,Z)}.
\item
Our arguments prove an analogous result 
for any lattice $\Gamma$ in place of $\SL(2,\ZZ)$. In this case, the quantity 
$\hat{\xi}(\vu,\tau_1,\tau_2)$ is replaced by the supremum of the
distance between a fixed reference point in
$\gmg$ and $\Psi(\vu/|\vu|)a_s$, where $s \in [\tau_1, \tau_2]$ --- see
\S\S7-8. Also,
the $\delta_0$ can be taken to be $\delta_\Gamma/24$, 
where $\delta_\Gamma$ is defined as explained in Remark \ref{remark
1}(1). 
\end{enumerate}
\end{remark}
 
For $\beta>0$, say $z \in \RR$ is $\beta$-diophantine if there is
$c>0$ such that $|z-p/q|\geq
c q^{-\beta}$ for all $p,q \in
\ZZ$. 
Note that quadratic irrationals are
$2$-diophantine, and by Roth's theorem, all algebraic numbers are
$2+\vre$-diophantine for any $\vre>0$, like Lebesgue almost any
real number. 
In
the following, we extend the definition to the case $\beta=+\infty$
with the convention that every number (even rationals) is
$\infty$-diophantine.
It is well-known (see Lemma \ref{lem:
betadiophantine}) that when $z$ is $\beta$-diophantine,
$\hat{\xi}_{z, \tau_1, \tau_2} $ can be bounded in terms of
$\beta$. This yields:
\begin{cor}
\name{cor: main, SL(2,Z)}
Given $\beta \in [2,+\infty]$, $\theta \in (0,1]$ and
$-\frac{1}{\beta-1} < \alpha <1$, there is $\delta>0$ such 
that for any $\vu \in \PP$  a vector with a $\beta$-diophantine slope and any compactly supported
$\theta$-H\"older function $f$ on $\PP$ there are positive $T_0$ and $C$ such that
for all $T>T_0$, 
\eq{main cor1}{
\left |\frac1{T^{1+\alpha}} \sum_{\gamma \in \Gamma_T}
 f\left(\frac{\gamma \vu}{T^{\alpha}} \right) - \frac{2}{\mu(\gmg)} \int_{\PP}
\frac{f(\vv)}{\vv \star \vu}\, d\vv \right | < CT^{-\delta} .}
\end{cor}

\begin{remark} In a recent and independent work \cite{no2}, Nogueira proved a similar result, with a better estimate of the error term, but in the particular case in which $f$ is the characteristic function of a square, and the norm $||\cdot||$ is the supremum norm, generalizing his previous results \cite{no}. The method used is completly different from ours.
\end{remark}

Applying Theorem \ref{thm: main, SL(2,Z)} to the `shrinking target problem', we obtain:
\begin{cor}
\name{cor : Dirichlet}
Let $\delta_0$ be as in Theorem \ref{thm: main, SL(2,Z)}, and let
$\vv, \vu
\in \PP$. Then:
\begin{enumerate}
\item
If $\vu$ has 
$\beta$-diophantine slope, then there are positive constants $C$ and
$T_0$ such that for all $T \geq 
T_0$ there is $\gamma \in \Gamma_T$ such that 
 $$|\gamma \vu -\vv| < CT^{-\frac{2\delta_0}{3\beta}}.$$

\item
If the slope of $\vu$ is irrational then there is a positive constant
$C$ such that there are
infinitely many $\gamma \in \Gamma$ solving 
 $$|\gamma \vu -\vv| < C||\gamma||^{-\frac{\delta_0}{3}}.$$
\end{enumerate}
\end{cor}

\subsection{Notation}
Throughout this paper the Vinogradov symbol $A \ll B$ means that there
is a constant $C$ such that $A \leq CB$, where $A$ and $B$ are expressions
depending on various quantities and the {\em implicit constant} $C$ is
independent of these quantities. In particular, throughout the paper
the implicit constant may depend on $\Gamma$, on the choice of
the norm $\| \cdot \|$, on auxilliary functions $\Psi, \, \phi$, but {\em not on the
function $f$ nor the initial 
point $\vu$.} The notation $A \asymp B$ means that $A \ll B$ and $B\ll A$.

\section{The norm estimate}
\subsection{The setup}
 Let $G$ act on $\RR^2$ by matrix multiplication on the left.
 Define the following matrices
 $$h_s=\left[ \begin{array}{cc}
      1 & s \\
      0 & 1 \\
     \end{array} \right],
     \;
      a_t=\left[ \begin{array}{cc}
      e^{t/2} & 0 \\
      0 & e^{-t/2} \\
     \end{array} \right], 
\;
 r_\theta=\left[ \begin{array}{cc}
      \cos \theta & -\sin \theta \\
     \sin \theta & \cos \theta \\
     \end{array} \right]
.
     $$

 The stabilizer of  $\vu_0=\left[ \begin{array}{c} 1 \\ 0
\end{array}\right]$ is precisely the unipotent subgroup 
  $$ H=\{ h_s  \, : \, s \in \RR \},  $$
so that $\PP$ is identified with the quotient $G/H$ via
the map 
$gH \mapsto g\vu_0$. Let $\Gamma$ be a lattice in $G$, and let $\tau: G
\to G/H$ and $\pi: G \to \gmg$ be the natural quotient maps. 
\[
\xymatrix{
& G \ar[dl]_{\pi} \ar[dr]^{\tau}  \\
\gmg & & G/H}
\]

We 
define a haar measure $\lambda$ on $H$ by $d\lambda(h_s) = ds$. 

\subsection{The section}
Define $\Psi : \PP \rightarrow G$ by

\eq{eq: defn Psi}{
 \Psi\left(\left[ \begin{array}{c} x \\ y \end{array}\right]\right)=
  \left[
     \begin{array}{cc}
      x & -y/(x^2+y^2) \\
      y & x/(x^2+y^2) \\
     \end{array}
    \right],
}

The function $\Psi$ is a section in the sense
that $\tau \circ \Psi = \mathrm{Id}|_{G/H}$, i.e. for all $\vu\in \PP,$ 
\eq{equ1}{
   \Psi(\vu)\vu_0=\vu.}
The following equation is easily verified:
\eq{equ4}{
    \vv \star \vu= \Biggl\| \Psi(\vv) \left[ \begin{array}{cc}
      0 & 1 \\
      0 & 0 \\
     \end{array} \right] \Psi(\vu)^{-1} \Biggr\|.
}
Note that \equ{equ4}
does not depend on the choice of the section $\Psi$ (as might not be
obvious from the formula).
It can be also checked that for any $t\in \RR$ and $\vv\in \PP$, we have
\eq{equ2}{
    \Psi(e^t \vv)=\Psi(\vv) a_{2t}.
}

Define
\eq{eq: defn D}{
D=D(\vu, f) = \sup_{x \in \supp \, f} \|\Psi(x)\Psi(\vu)^{-1}\|,
}
this quantity satisfies
\eq{eq: sameD}{
D(\vu,f) \asymp D_0(\vu,f), 
}
where $D_0$ is as in \equ{eq: defn D0}. 
Indeed, let $x \in \supp \, f$, then 
$$\|\Psi(x)\Psi(\vu)^{-1}\|=\left\| \Psi\left( \frac{x}{|x|}\right) a_{2\log(|x|/|\vu|)}\Psi\left( \frac{\vu}{|\vu|}\right)^{-1} \right\| $$
$$\asymp \| a_{2\log(|x|/|\vu|)} \| \asymp \max \left( \frac{|x|}{|\vu|}, \frac{|\vu|}{|x|} \right).$$

\subsection{The cocycle}
Let $\vu\in \PP$ and $g \in G$, define $c_{\vu}(g)$
by the following implicit equation: 
\eq{equ3}{
\left[
\begin{array}{cc}
      1 & c_{\vu}(g) \\
      0 & 1 \\
     \end{array}
 \right]
= \Psi(g \vu)^{-1}g \Psi(\vu).
}
This makes sense because the right hand side stabilizes $\vu_0$, so is in
$H$. It is easily checked that $c$ is a cocycle, meaning it satisfies
for any $g_1,g_2 \in G$ and $\vu\in \PP$, 
$$c_{\vu}(g_1g_2)=c_{g_2 \vu}(g_1)+c_{\vu}(g_2).$$

One also sees that in terms of Iwasawa decomposition, we have
\eq{eq: Iwasawa}{
g = kan \ \ \implies \ \ ka = \Psi(g \vu_0),\ n=h_s \ \mathrm{where} \ s
= c_{\vu_0}(g).
}
Therefore we can write haar measure $\mu$ on
$G$ by the formula 
\eq{eq: haar normalization}{
d\mu(g) = d\tau(g)\, d\lambda(c_{\vu_0}(g)).
}
Note that the normalization of Lebesgue measure on $\RR^2$ 
and $\lambda$ determine
a normalization for $\mu$. 
Changing the section $\Psi$ gives rise to a homologous
cocycle, so that $\mu$ is actually independent of $\Psi$. 
It follows from \equ{eq:
Iwasawa} that
\eq{equ5}{
    c_{\vu_0}(gh_s)=c_{\vu_0}(g)+s,
\ \ \ \ 
    c_{\vu_0}(ga_t)=e^{-t} c_{\vu_0}(g),
}
and from \equ{equ1} that
\eq{eq: from section}{
c_{\vu}(g)=c_{\vu_0}(g \Psi(\vu)).
}

\begin{lem}\name{lemme1 revisited}
Let $D= D(\vu, f)$ be as in \equ{eq: defn D}. 
For any $f \in C_c(\PP)$,  any $\vu \in \PP$ and any $g \in
G$ for which $g \vu \in \supp \, f$ we have 
\eq{eq: uniformity revisited}{
  \Bigl| \|g\| - |c_{\vu}(g)| \, (g \vu \star \vu) \Bigr|
\leq D.
}
\end{lem}

\begin{proof}
By \equ{equ3}, we have
\[
\begin{split}
g & = \Psi(g\vu)\left( \mathrm{Id} + \left [\begin{matrix} 0 & c_{\vu}(g)
\\ 0 & 0 
\end{matrix} \right ] \right )\Psi(\vu)^{-1} \\
& = \Psi(g\vu) \Psi(\vu)^{-1} +
c_{\vu}(g) \Psi(g\vu) \left[\begin{matrix} 0 & 1 \\ 0 & 0  \end{matrix}
\right] \Psi(\vu)^{-1}.
\end{split}
\]

By \equ{equ4},
$$\left\|c_{\vu}(g) \Psi(g\vu) \left[\begin{matrix} 0 & 1 \\ 0 & 0  \end{matrix}
\right] \Psi(\vu)^{-1} \right \| = |c_{\vu}(g)| \,  (g \vu \star \vu).$$
The claim follows. 
\end{proof}

\subsection{Some useful inequalities}

 Here we state and prove elementary inequalities that will be useful
later. The first remark is that the $\star$-product is
well-approximated by the
product of the norms:  
\eq{eq: starproduct}{
 \vv \star \vu \asymp |\vv| \, |\vu|.
}
Indeed,
$$\vv \star \vu=\Biggl\| \left[
     \begin{array}{cc}
      -u_2 v_1 & u_1v_1 \\
      -u_2 v_2 & u_1v_2 \\
     \end{array}
    \right]\Biggr\| \asymp \max_{i,j=1,2} \{ |u_i v_j |\} =
\max_{i=1,2} \{ | u_i | \}  \max_{j=1,2} \{ | v_j | \} .$$

The following upper bound will also prove helpful: for any
$\theta$-H\"older compactly supported $f: \PP \to \R$ 
\eq{eq: intnosolarge}{
\left| \int_{\PP} \frac{f(\vv)}{\vv \star \vu} \, d\vv \right| \ll \|f\|_{\theta} \frac{R(f)}{|\vu|}. 
}
This is proved as follows:
$$\left| \int_{\PP} \frac{f(\vv)}{\vv \star \vu} \, d\vv \right| \stackrel{\equ{eq: starproduct}}{\ll}
\left| \int_{\supp \, f} \frac{f(\vv)}{|\vv| \, |\vu|} \, d\vv \right| \ll |\vu|^{-1} \left( \int_\PP f^2(\vv) |\vv|^{-2} \, d\vv  \right)^{1/2} \left( \int_{\supp \, f} d\vv \right)^{1/2},$$
by the Cauchy-Schwarz inequality, and since $\supp \, f$ is in a disk
of radius approximately $R(f)$, this implies \equ{eq: intnosolarge}.

\section{From the plane to the homogeneous space}
In this section we pass from a function $f: \PP \to \RR$
to a function $\bar{f}
: \gmg \to \RR$. 
This is done in two steps: lifting to 
a function $\til f$ using the section $\Psi$ and a bump function; and
summing along $\Gamma$-orbits to obtain a function $\bar{f}$ on
$\gmg$.

Assume $f$ is compactly supported and non-negative.
Fix $\phi : \RR \rightarrow \RR$ a nonnegative $C^\infty$ function,
vanishing outside $[-1,1]$, such that $\int_\RR \phi(t)dt =1$.
Set 
\eq{eq: defn til f}{
\til f
: G \to \RR, \ \ \til f(g) = f(\tau(g))
\phi
(c_{\vu_0}(g))
}
(a compactly supported smooth function on $G$) and 
\eq{eq: defn bar f}{
\bar{f}
(x)= \sum_{g \in \pi^{-1}(x)} \til f
(g)
}
(a finite sum for each $x$). The normalization \equ{eq: haar
normalization} for $\mu$ 
ensures that 
\eq{eq: normalization again}{
\int_{\gmg}
\bar{f} \, d\mu = \int_{\RR^2} f(x)\, dx.
}

The distribution of the cloud $\Gamma_T$ turns out to be linked with
the norm $\vv \mapsto \vv \star \vu$, and for this reason we will have
to work with the more precise measures of the support of $f$: 

\[r^{(\vu)}(f)  = \inf_{\vv \in \supp \, f } \vv \star \vu \asymp r(f) \, |\vu|, \ \ 
\ \ \ \ \  
R^{(\vu)}(f)=\sup_{\vv \in \supp \, f } 
\vv\star \vu \asymp R(f) \, |\vu|,
\]
\[   v^{(\vu)}(f)=
\frac{R^{(\vu)}(f)}{r^{(\vu)}(f)} \asymp v(f). \\ 
\]

\begin{lem} \name{lem: boundary effect again}
Let $\vu \in \PP,\, \til \vu = \Psi(\vu), \, \gamma \in \Gamma, \, f
\in C_c(\PP)$ and let 
$\til f$ be as
in \equ{eq: defn til f}. 
Let $D=D(\vu, f)$ be as in \equ{eq: defn D}.

Then 
\eq{eq: for boundary effect}{
\begin{split}
r\leq
r^{(\vu)}(f), \ \| \gamma\|\leq T & \ \implies \ 
\int_{-\left(1+(T+D)/r\right)}^{1+(T+D)/r}  \til f
(\gamma \til \vu h_s)\, ds = f(\gamma \vu), \\
R^{(\vu)}(f) \leq R, \ \|\gamma \| \geq T & \ \implies \   \int_{-\left((T-D)/R -1 
\right)}^{(T-D)/R-1}  \til f
(\gamma \til \vu h_s)\, ds=0.
\end{split}
}
\end{lem}

\begin{proof}
Since $\int \phi(s) ds =1$ and $\supp \, \phi \subset [-1,1]$, for the
first claim it suffices to show that for 
$\| \gamma \|\leq T
$ and $\gamma \vu \in \supp \, f$ we have 
$$[-1,1] \subset \left\{c_{\vu_0}(\gamma \til \vu h_s): |s| \leq 
1+(T+D)/r \right\},$$
or by \equ{equ5}, that 
$$|c_{\vu_0}(\gamma \til \vu)| \leq \frac{T+ D}{r}.$$
This follows from \equ{eq: from section} and \equ{eq: uniformity
revisited}. The proof of the second claim is similar.  
\end{proof}
 
We will need to control the H\"older norm of $\bar{f}$ in terms
of that of $f$. 
\ignore{
We begin with some general remarks
concerning Sobolev and H\"older norms. Let $q \in \Z_+$ and $p \geq 1$. Suppose $M$
is a $d$-dimensional manifold equipped with an absolutely continuous 
smooth measure $m$. A {\em ($C^q$-) framing} of $M$ is a choice of $d$ ($C^q$-)
smooth  
vector fields $X_1, \ldots, X_d$, which are linearly independent at
each point. By the corresponding Sobolev $p,q$
norm on the space of compactly supported smooth functions on $M$, we
mean 
$$\|f\|_{p,q} = \sum_{|I| \leq q} \left(\int \left|X^I f\right|^p \, dm
\right)^{1/p},$$ 
where the sum ranges over multi-indices $I=(i_1, \ldots, i_d)$, $|I| =
i_1 + \cdots + i_d$, and $X^I = X_1^{i_1} \circ \cdots \circ
X_d^{i_d}$, $X^k_i$ denotes $k$ applications of $X_i$
to a $C^q$ function $f$. This definition depends on the compact set,
and on the order chosen for the $X_i$.  
 Let $X_1,X_2,X_3$ be a basis of the Lie algebra of $G$. Risking
confusion, we use the same letter to denote the framing of $G$
obtained by left-transporting this basis, i.e.  
 $$X_i f (g) = \left.\frac{d}{dt}\right|_{t=0} f \left(g\exp(tX_i) \right).$$ 
 This defines our Sobolev norms on $G$. By their definition, the
vector fields $X_i$ descend to vector fields $\bar{X}_i$ on the
quotient $\gmg$, which defines in turn a Sobolev norm on $\gmg$. Let
$d$ be a Riemannian distance on $G$ which is left-$G$-invariant, it
defines in turn a distance on $\gmg$. 
}
For $\theta \in (0,1]$, define a
H\"older norm on compactly supported function $f$ on $G$ or $\gmg$: 
 $$|| f ||_\theta = \sup_{\dist(x,y)\leq1} \frac{|f(x)-f(y)|}{\dist(x,y)^\theta}.$$ 
Here by dist we denote a left-invariant Riemannian metric on $G$, or
the corresponding metric induced on $\gmg$. 
\begin{lem} \name{lem:holdercontrol}
For any $\sigma > 1$ and any $\theta \in (0,1]$ there is a constant
$c=c_{\sigma, \theta}>0$ and a compact set $K_\sigma\subset
\gmg$ such that for any $\theta$-H\"older compactly supported function
$f$ with
$$\supp \, f \subset A_\sigma=\{\vw \in \PP \,: \, \sigma^{-1} \leq |\vw| \leq \sigma \},$$
 we have $\supp \, \bar{f} \subset K_\sigma$ and
 
\eq{eq: holdercontrol}{
\|\bar{f} \|_{\theta} \leq  c \| f\|_{\theta}.
}
 
\end{lem}
\begin{proof}
 We first prove that for some constant $c_1>0$, $$\|\til{f} \|_{\theta} \leq  c_1 \| f\|_{\theta}.$$ 
 Note that since $\til f(g) = f(\tau(g)) \phi (c_{\vu_0}(g))$, the
support of $\til f$ is contained in the compact set
$\til K_\sigma=\tau^{-1}(A_\sigma) \cap c_{\vu_0}^{-1}([-1,1])$, and that
$\tau$, $\phi\circ c_{\vu_0}$ are Lipschitz when restricted to
$K_\sigma$, so 
   $$\|\til{f} \|_{\theta} \ll \| f\|_{\theta}.$$ 
 Also note that $\# \, \til K_\sigma \cap \pi^{-1}(x)$ is bounded
independently of $x$ by compactness of $\til K_\sigma$, by a bound
depending on $\sigma$ only. Thus, 
   $$\|\bar{f} \|_{\theta} \ll \| \til{f} \|_{\theta}.$$   
   We put $K_\sigma=\pi(\til K_\sigma)$. 
 \end{proof}


\section{Radial Partition of unity}
We will use a partition of unity to
reduce to the case when $\supp \, f$ is contained 
in a narrow annulus around zero. Let $\kappa$ be the `tent' map

$$\kappa(x)=\begin{cases} 0 & \text{if $x\leq-1$ or $x\geq 1$}, \\
x+1 & \text{if} \ -1\leq x \leq 0,\\
1-x & \text{if} \ 0\leq x \leq 1,\\
\end{cases}
$$
which is a $1$-Lipschitz map and satisfies for all $x$
\eq{eq: partition}{
\sum_{\ell \in \Z} \kappa (x+\ell) =1.
}

Now given a parameter $\alpha \geq 1$, for any H\"older function $f$
on $\PP$ we define for $\ell \in \Z$ 
$$f_\ell(\vv) = f_{\ell}^{(\alpha)}(\vv) = f(e^{-\ell/\alpha} \vv)
\cdot \kappa \left( \alpha \log \frac{  \vv \star \vu}{|\vu |} \right),$$
so that for all $\vv \in \PP$
\eq{eq: f as sum}{
f(\vv) = \sum_{\ell \in \Z} f_{\ell}(e^{\ell /\alpha}\vv).
}

 If $f_\ell$ is not identically zero, there exists $\vv \in \PP$ in its support. Then we have
 $e^{-\ell/\alpha} (\vv \star \vu) \in [r^{(\vu)}(f), R^{(\vu)}(f)]$
and  $\vv \star \vu\in [e^{-1/\alpha}|\vu|,e^{1/\alpha}|\vu|]$, so 
\eq{eq: ellbounds}{
e^{-1/\alpha} \frac{r^{(\vu)}(f)}{|\vu|} \leq e^{-\ell/\alpha} \leq e^{1/\alpha} \frac{R^{(\vu)}(f)}{|\vu|} 
}
Note that this implies that the number of
nonzero summands in 
\equ{eq: f as sum} is at most $\alpha \log v^{(\vu)}(f)+2.$ 

 The properties of the maps $f_\ell$ are summarized in the following Lemma.

\begin{lem} \name{lem:fixsigma}
1) For all $\ell$, 
\eq{eq: thin support}
{e^{-1/\alpha}|\vu|\leq r^{(\vu)}(f_\ell)\leq R^{(\vu)}(f_\ell) \leq e^{1/\alpha}|\vu|.
}
2) There exists $\sigma>1$ such that for all $\alpha \geq 1$ and all $\ell$, 
$$\supp \, f_\ell  \subset A_\sigma.$$
3) 
\eq{eq: nodistortion}{ 
D(e^{\ell/\alpha}\vu,f_\ell)\leq D(\vu,f),
}
4)
 \eq{eq: estimateofholder}{
\| f_\ell \|_{\theta}  \ll  \alpha \|f\|_{\theta}.
}

5) Let $r_{\ell}=e^{(\ell-1)/\alpha}|u| \leq
r^{(e^{\ell/\alpha}\vu)}(f_{\ell}),\ R_{\ell} =
e^{(\ell+1)/\alpha}|\vu| \geq R^{(e^{\ell/\alpha}\vu)}(f_{\ell}).$
Then  
\eq{eq: intup}{
\sum_{\ell} r_\ell^{-1}\int_\PP f_\ell \leq e^{2/\alpha} \int_\PP
\frac{f(\vv)}{\vv \star \vu} d\vv,  
}
and
\eq{eq: intlow}{
  \sum_{\ell} R_\ell^{-1}\int_\PP f_\ell \geq e^{-2/\alpha} \int_\PP
\frac{f(\vv)}{\vv \star \vu} d\vv. 
}

\end{lem}
\begin{proof}
 The first property is a direct consequence of the fact that $\supp \, \kappa \subset [-1,1]$. 
 The second one follows easily from \equ{eq: thin support} and \equ{eq: starproduct}. 
 To prove the third statement, notice that the definition of
$f_\ell$ also implies that $\supp f_\ell \subset e^{\ell/\alpha} \supp
f$. Together with \equ{eq: defn D} and \equ{equ2}, this
gives the desired result.
 
 We now proceed to the proof of \equ{eq: estimateofholder}. The first
two summands in \equ{eq: Holder norm} for $f_\ell$ are clearly
controlled by $\|f\|_\theta$, so we need to show that
$$\sup_{0<|x-y|\leq |x|/2} \frac{ |x|^\theta
|f(x)-f(y)|}{|x-y|^\theta} \\ \|f\|_\theta.$$
Let
$\kappa_\alpha(x)=\kappa(\alpha \log \frac{x\star \vu}{|\vu|})$, and
let $x,y \in \PP$ such that $|x-y|\leq |x|/2$. Then  
  $$|\kappa_\alpha(x)-\kappa_\alpha(y)| \leq \alpha \left| \log
\left(\frac{x}{|x|} \star \frac{\vu}{|\vu|}\right)  
 - \log \left(\frac{y}{|x|} \star \frac{\vu}{|\vu|}\right) \right|,$$
 so, since $\log$ and $\star$ are Lipschitz function when restricted to compact sets,
 $$|\kappa_\alpha(x)-\kappa_\alpha(y)| \ll \alpha \frac{|x-y|}{|x|}.$$ 
We have
 $$|f_\ell(x)-f_\ell(y)| \leq |\kappa_\alpha
(x)(f(e^{-\ell/\alpha}x)-f(e^{-\ell/\alpha}y))| 
 + |(\kappa_\alpha (x)-\kappa_\alpha(y) )f(e^{-\ell/\alpha}y)|,$$
so that
 $$|f_\ell(x)-f_\ell(y)| \ll \frac{|x-y|^\theta}{|x|^\theta} || f
||_\theta + \alpha \frac{|x-y|}{|x|} || f ||_\theta $$ 
 and since  $\frac{|x-y|}{|x|}\leq \frac{|x-y|^\theta}{|x|^\theta}$ and $\alpha+1\ll \alpha$,
 this concludes the proof of \equ{eq: estimateofholder}.\\
 
 Let us prove \equ{eq: intup}. A change of variable $\vw=e^{-\ell/\alpha}\vv$ gives
 $$\sum_{\ell} r_\ell^{-1} \int_\PP f_\ell = \sum_{\ell} \int_\PP
\frac{e^{(1+\ell)/\alpha}}{|\vu|} f(\vw)\kappa\left(\alpha \log
\frac{\vw\star \vu}{|\vu|}+\ell \right) d\vw,$$ 
 but since $ \vw \star \vu \leq e^{(-\ell+1)/\alpha}|\vu|$ for every
$\vw \in e^{-\ell/\alpha}\supp f_\ell$ because 
of \equ{eq: thin support}, we have
$$\sum_{\ell} r_\ell^{-1} \int_\PP f_\ell \leq \int_\PP
\frac{e^{2/\alpha}}{\vw \star \vu} f(\vw)\sum_\ell\kappa\left(\alpha
\log \frac{\vw\star \vu}{|\vu|}+\ell \right) d\vw.$$ 
Using \equ{eq: partition} yields the required result. The proof of \equ{eq: intlow} is similar.
\end{proof}


\section{Effective equidistribution}
It was proved by Furstenberg that the horocycle flow on $\gmg$ is
uniquely ergodic when $\Gamma$ is cocompact, in particular every orbit
is uniformly distributed. We will need a strengthening due to Burger 
\cite[Thm. 2(C)]{bu}, which gives an
effective rate for the  convergence of ergodic averages. Denote by
$\|F\|_{p,q}$ the $L^p$-Sobolev norm on compactly supported
continuous functions involving all derivatives up to order $q$ (see
e.g. \cite{st} for definitions and some generalities concerning these
norms).

\begin{thm}[Burger]\name{thm: Burger}
For any cocompact lattice $\Gamma$ there are positive $\delta =
\delta_{\Gamma}$ and $c$ such that for any $S\geq 1$, any
$C^3$-map $F$ on $\gmg$ and any $x \in\gmg$,  
\[
\left| \frac{1}{2S} \int_{-S}^S F(x h_s) ds -
\frac1{\mu(\gmg)} \int_{\gmg} F\, d\mu
\right| \leq c \|F\|_{2,3} S^{-\delta}.
\]
\end{thm}

 It will be more convenient for us to work with H\"older norms, so we
will prefer the following Corollary: 
 
\begin{cor}\name{cor: Burger}
For any cocompact lattice $\Gamma$ there are positive
$\delta_{\Gamma}$ and $c$ such that for any $S\geq 1$, any 
$\theta$-H\"older map $F$ on $\gmg$ and any $x \in\gmg$,  
\eq{eq: Burger}{
\left| \frac{1}{2S} \int_{-S}^S F(x h_s) ds -
\frac1{\mu(\gmg)} \int_{\gmg} F\, d\mu
\right| \leq c \|F\|_{\theta} S^{-\theta \delta_\Gamma/7}.
}
\end{cor}
\begin{proof}
 By convolution of $F$ with a function of support in a ball of small
radius $\vre$ and $i$-th derivatives bounded by a multiple of
$\vre^{-i-3}$, one can approximate the $\theta$-H\"older map $F$
by a smooth map $F_\vre$ such that 
 $$\sup | F_\vre -F  | \leq \|F\|_\theta \vre^\theta,$$
 and
 $$\|F_\vre\|_{2,3}\ll \|F\|_\theta \vre^{-6}.$$
 The exponent $6$ here corresponds to 3 (the dimension) plus 3 (the number of derivatives). 
  So 
\begin{multline*}
\left| \frac{1}{2S} \int_{-S}^S F(x h_s) ds -
\frac1{\mu(\gmg)} \int_{\gmg} F\, d\mu
\right| \leq 
\left| \frac{1}{2S} \int_{-S}^S F_\vre(x h_s) ds -
\frac1{\mu(\gmg)} \int_{\gmg} F_\vre\, d\mu
\right| \\
+\left| \frac{1}{2S} \int_{-S}^S (F-F_\vre)(x h_s) ds -
\frac1{\mu(\gmg)} \int_{\gmg} (F-F_\vre)\, d\mu
\right| \ll \|F\|_\theta \vre^{-6} S^{-\delta_\Gamma} + \|F\|_\theta \vre^\theta.
\end{multline*}
Taking $\vre=S^{-\delta_\Gamma/(\theta+6)}$ gives a bound of
$\|F\|_\theta S^{\theta \delta_\Gamma / (6 + \theta)} \leq
\|F\|_\theta S^{\theta \delta_\Gamma / 7}$.
\end{proof}

In the non-uniform case use an analogous result of
Str\"ombergsson  \cite{st} (see also \cite{ff} for more
detailed results regarding the deviation of ergodic averages). 
We let 
\eq{eq: defn xi1}{
\xi(x,t)=e^{\dist(xa_{t},\pi(\Psi(\vu_0)))},}
where dist is a metric on $\gmg$ induced by a left-invariant
Riemannian metric on $G$. We fix a parameter
$\sigma$ as in Lemma \ref{lem:fixsigma}(2) and let $K_\sigma$ be a
compact subset of $\gmg$ as in Lemma \ref{lem:holdercontrol}.

\begin{thm}[Str\"ombergsson. Flaminio-Forni]\name{thm: Strombergsson}
For any lattice $\Gamma$ in $G$ there are positive $\delta =
\delta_{\Gamma}$ and $c$ such that for any $S\geq 1$, any
$C^4$-map $F$ on $\gmg$ supported on $K_\sigma$ and any $x \in\gmg$,  
\eq{eq: Strombergsson1}{
\left| \frac{1}{2S} \int_{-S}^S F(x h_s) ds -
\frac1{\mu(\gmg)} \int_{\gmg} F\, d\mu
\right| \leq c \|F\|_{2,4} S^{-\delta} \xi(x,\log S)^{\delta}.
}
\end{thm}
\begin{proof} Let us indicate briefly how to recover \equ{eq: Strombergsson1} from \cite[Theorem 1]{st}.
We will use Str\"ombergsson's notations. For any $S\geq 10$ and 
any parameter $\alpha \in [0,\frac12)$, we have: 

\begin{multline*}
\frac{1}{S}\int_0^S F(x h_s) ds =  \frac1{\mu(\gmg)} \int_{\gmg} F\, d\mu + \\
 O(\| F\|_{2,4}) \left(  r^{-\frac12}\log^3{(r+2)}+  r^{s_1^{(j)}
-1}+S^{s_1-1}  \right) + O(\| F\|_{N_\alpha}) r^{-\frac12}, 
\end{multline*}
where $r(x,S)=S/\xi(x,\log S)$, $\| . \|_{N_\alpha}$ is a weighted
supremum norm and $s_1^{(j)}>0$. The parameter $\alpha$ is chosen to
be zero, and we have  
$\|F\|_{N_\alpha} \ll  \|F\|_{2,4}$, since $F$ is supported on $K_\sigma$.
It can be checked that $r(x,S) \asymp r(x h_{-S},2S)$, and this
combined with the fact that 
$$\frac{1}{2S} \int_{-S}^S F(x h_s) ds = \frac{1}{2S} \int_{0}^{2S}
F(x h_{-S} h_s) ds,$$ 
proves the claim.
\end{proof}

 In the case of $\SL(2,\ZZ)$, it is a classical fact that one can take any $\delta_\Gamma<1/2$. 
 \comfm{The only reference I have on this is an article in german of
Walter Roelcke in 1953 that I did not read. Do you have better ?
Apparently this is due to Selberg but I am unable to say where in his
work it is proved.} \combarak{After a discussion with Shahar Mozes it
seemed like a good place to look is the article {\em Introduction to
Kloostermania} by Huxley in Banach Ctr. Publications No. 17. However I
could not find it in our library. Can you find it? There is also some
information in the paper 
A. Gamburd: Spectral gap for infinite index "congruence'' subgroups of
SL(2, Z), Israel Journal of Mathematics 127 (2002), 157-200, but that
is certainly not the right refernce.
}

The following Corollary is proved the same way as Corollary \ref{cor: Burger}.

\begin{cor}\name{cor: Strombergsson}
For any lattice $\Gamma$ in $G$ there is a positive $\delta =
\delta_{\Gamma}$ such that for any $\theta \in (0,1]$, there exists a
positive $c$ such that for any $S\geq 1$, any $\theta$-H\"older map
$F$ on $\gmg$ supported on $K_\sigma$ and any $x \in\gmg$,

\eq{eq: Strombergsson}{
\left| \frac{1}{2S} \int_{-S}^S F(x h_s) ds -
\frac1{\mu(\gmg)} \int_{\gmg} F\, d\mu
\right| \leq c \|F\|_{\theta} S^{-\theta \delta_\Gamma/8 } \xi(x,\log
S)^{\theta \delta_\Gamma/8}. 
}
\end{cor}

\section{Proof of Theorem  \ref{thm: main, cocompact}} 
 Writing $f$ as the sum of a nonnegative and a nonpositive function,
it is sufficient to prove the Theorem under the assumption that $f$ is
nonnegative. 
Let $\delta_{\Gamma}$ be as in Theorem \ref{thm: Burger} and let $\delta_0
= \delta_\Gamma/21.$ 
Given $f$ and $\vu$, let $D = D(\vu, f)$ be as in \equ{eq: defn D}. In
view of \equ{eq: sameD} it suffices to prove the Theorem with $D$ replacing
$D_0$. Let $\alpha \geq 1$ a parameter that will be fixed later, and
take a radial partition of unity 
$f = \sum f^{(\alpha)}_{\ell}.$ Thus the $f_{\ell} =
f^{(\alpha)}_\ell$ are nonnegative $\theta$-H\"older functions for which, by
\equ{eq: thin support},
\eq{eq: choice of Ri}{
r_{\ell}=e^{(\ell-1)/\alpha}|\vu| \leq
r^{(e^{\ell/\alpha}\vu)}(f_{\ell}),\ \ \ R_{\ell} =
e^{(\ell+1)/\alpha}|\vu| \geq R^{(e^{\ell/\alpha}\vu)}(f_{\ell}). 
}
Hence for any $\ell$ with $f_\ell$ nonzero, we have 
\eq{eq: boundonrl}{
r_\ell \stackrel{\equ{eq: ellbounds}}{\ll} \frac{|\vu|}{r(f)} \ll D
, \ \ \mathrm{and} \;\; \frac{1}{r_\ell}\ll \frac{R(f)}{|\vu|}, 
}
and $R_\ell =e^{2/\alpha}r_\ell \leq e^2 r_\ell \ll D.$
Let $c>1$ be such that $R_\ell \leq cD/3 $, 
fix $T_0=cD$ as in \equ{eq: choice of T0}, so that
\eq{eq: T large enough}{
T_0 \geq  \frac{|\vu|}{R(f)},
}
and consider any $T\geq T_0$. Let $\til u = \Psi(\vu)$ and $x_0 = \pi(\til
u)$. Fix the value of $\alpha$ to be
$$\alpha=\left( \frac{R(f) T}{|\vu|}  \right)^{\theta
\delta_0}\stackrel{\equ{eq: T large enough}} {\geq} 1,$$ 
 define $\til f_{\ell}$ and $\bar f_{\ell}$ by \equ{eq: defn til
f}, \equ{eq: defn bar f}, and set 
$$
\til u = \Psi(\vu), \ x_0 = \pi(\til u), \ \ \mathrm{and } \ 
\eta = 1- \theta \delta_\Gamma/7,$$
so that 
\eq{eq: property of alpha}{
\alpha^2 T^{\eta} =  \left(\frac{R(f)}{|\vu|} \right)^{2\theta \delta_0} T^{1-\theta
\delta_0} , \ \ \  \frac{T}{\alpha} = \frac{|\vu|}{R(f)} T^{1-\theta \delta_0}.  
}
Then for an upper bound we have:
\[
\begin{split}
S_{f,\vu}(T)  & = \sum_{\gamma \in \Gamma_T} f(\gamma \vu) = \sum_{\ell}
\sum_{\gamma \in \Gamma_T} f_{\ell}(e^{\ell/\alpha} \gamma \vu) \\ 
& \stackrel{\equ{eq: for boundary effect},\equ{equ2}}{=} \sum_{\ell} \sum_{\gamma \in
\Gamma_T} \int_{-\left(1+(T+D)/r_{\ell}\right)}^{1+(T+D)/r_{\ell}}  \til f_{\ell}
(\gamma \til \vu a_{2\ell/\alpha}h_s)\, ds \\
& \stackrel{f_{\ell} \geq 0}{\leq} \sum_{\ell}
\int_{-\left(1+(T+D)/r_{\ell}\right)}^{1+(T+D)/r_{\ell}}
\bar{f}_{\ell} (x_0 a_{2\ell/\alpha} h_s)\, ds \\ 
& \stackrel{\equ{eq: Burger}}{\leq} \frac{2}{\mu(\gmg)} \sum_{\ell}
\left(\frac{T+D}{r_{\ell}} + 1 \right) 
\int_{\gmg} \bar{f}_{\ell} \, d\mu + c_1\sum_{\ell} \|\bar{f}_{\ell}\|_{\theta}
\left(\frac{T+D}{r_{\ell}}+1\right)^{\eta} \\
& \stackrel{\equ{eq: normalization again}, \equ{eq:
holdercontrol},
\equ{eq: boundonrl}}{\leq} \frac{2(T+c_2D)}{\mu(\gmg)}
\sum_{\ell} 
r_{\ell}^{-1} \int_{\PP} f_{\ell} \, dx 
+ c_3 T^{\eta}
\sum_{\ell} \|f_{\ell}\|_{\theta} r_\ell^{-\eta}
\left(1+ \frac{D}{T} +\frac{r_\ell}{T}\right)^{\eta}
\\
& \stackrel{\equ{eq: estimateofholder},
\equ{eq: intup}}{\leq} \frac{2(T + c_3D) }{\mu(\gmg)}  
e^{2/\alpha} \int_{\PP} \frac{f(\vv)}{\vv \star \vu} \, d\vv  
+  c_4 \alpha T^{\eta} \|f\|_{\theta} \left( \frac{R(f)}{|\vu|}\right)^{\eta}  
(\alpha \log v^{(\vu)}(f)+2)\\ 
& \stackrel{e^{2/\alpha}=1+O(1/\alpha) }{\leq}
\frac{2T}{\mu(\gmg)} \int_{\PP} 
\frac{f(\vv)}{\vv \star \vu} \, d\vv + c_5 \left( D+\frac{T}{\alpha}\right) \int_{\PP} 
\frac{f(\vv)}{\vv \star \vu} \, d\vv \\
& \; \; \; \; \ \ \ \ \ 
+ c_6 \alpha^{2} T^{\eta}\ (\log v^{(\vu)}(f)  +1)
\|f\|_{\theta} \left( \frac{R(f)}{|\vu|}\right)^{\eta}.
\end{split}
\]
Using the upper bound \equ{eq: intnosolarge} and \equ{eq: property of
alpha} we obtain
$$ 
S_{f,\vu}(T) - \frac{2T}{\mu(\gmg)} \int_{\PP} \frac{f(\vv)}{\vv \star \vu} \, d\vv  
\ll  \|f \|_{\theta}\frac{R(f)}{|\vu|} \left( D+T^{1-\theta \delta_0}
(\log v^{(\vu)}(f)
+1)\left(\frac{R(f)}{|\vu|}\right)^{-\theta\delta_0} \right), 
$$
as claimed. For the lower bound, the proof is very similar, with upper
bounds replaced by lower bounds, $r_\ell$ replaced by $R_\ell$, except
that in order to apply \equ{eq: Burger} to $\bar{f_\ell}$ for the time
$S=\frac{T-D}{R_\ell}-1$, one has to check that if $\ell$ is such that
$f_\ell$ is nonzero, then 
$$\frac{T-D}{R_\ell}-1\geq 1.$$
Since $R_\ell\leq cD/3$, and $T\geq T_0\geq cD$, we have
$$\frac{T-D}{R_\ell}-1 \geq \frac{cD - D}{cD/3} -1 \geq 1,$$
as required. 
\qed

\begin{proof}[Proof of Corollary \ref{cor: applications}]
Let $\delta=\min \left( 1-|\alpha|, \theta \delta_0(1+\alpha)
\right)>0$. 
We apply Theorem \ref{thm: main, cocompact} to $f$ and 
$$\vw= \vw(T)= \frac{\vu}{T^{\alpha}}.$$
Considering separately the cases $\alpha \geq 0$ and $\alpha \leq 0$,
we see that 
$$D_0(\vw, f) \leq c_0 T^{|\alpha|},$$
where $c_0$ is a constant
depending on $f$ and $\vu$. Note that
$\displaystyle{\frac{R(f)}{|\vw|}=T^\alpha \frac{R(f)}{|\vu|}},$ so
that 
$$B(\vw,f)=T^{-\alpha \theta \delta_0}B(\vu,f).$$
In order to
apply Theorem \ref{thm: main, cocompact}, we need to check 
\equ{eq: choice of T0}, i.e., that
$$T \geq c c_0T^{|\alpha|},
$$
which clearly holds for all large enough $T$.
Therefore there is a positive $C$ (depending on $f$ and $\vu$ but
independent of $T$) such that 
\[ \begin{split}
CT^\alpha\left (T^{|\alpha|}+T^{1-\theta\delta_0(1+\alpha)} \right)  &
> \left|S_{f, \vw}(T) - \frac{2T}{\mu(\gmg)} \int_{\PP} 
\frac{f(\vv)}{\vv \star \vw}\, d\vv \right| =\\ 
& \ \ \ \, \left|S_{f, \vw}(T) -
\frac{2T^{1+\alpha}}{\mu(\gmg)} \int_{\PP} 
\frac{f(\vv)}{\vv \star \vu}\, d\vv \right|.
\end{split} \]
Dividing through by $T^{1+\alpha}$ gives
$$\left| \frac{S_{f, \vw}(T)}{T^{1+\alpha}} -
\frac{2}{\mu(\gmg)} \int_{\PP} 
\frac{f(\vv)}{\vv \star \vu}\, d\vv \right| < C\left(
T^{|\alpha|-1}+T^{-\theta \delta_0(1+\alpha)}\right) \leq C' T^{-\delta}.$$ 
\end{proof}

\section{Diophantine properties}
 Let $x \in \gmg$, $0\leq s_1 \leq s_2$ be two real numbers. The
quantity 
 \[ \xi(x , s_1, s_2)
=\max_{s_1\leq s \leq s_2} \xi(x,s), \ \ \  \xi (x,s) \mathrm{\ as \
in \ } \equ{eq: defn xi1}\]
describes the excursions of the geodesic $xa_s$ into the cusps of
$\gmg$ for times $s \in [s_1, s_2]$. In the case $\Gamma=\SL(2,\ZZ)$
and $x=\pi \circ \Psi(\vv)$, 
one can relate 
$\xi(x,s_1,s_2)$ with the diophantine properties 
of the slope of $\vv$.   
 
\begin{lem} 
 Let $\vv=\left[ \begin{array}{c} \vv_x \\ \vv_y \end{array} \right]
\in \PP$ such that $|\vv|=1$. Then
$$ \xi\left(\pi \circ \Psi(\vv),s_1,s_2\right) \ll \hat{\xi}\left( \vv,
s_1,s_2 \right)$$ 
\end{lem} 
\begin{proof}
Clearly, for all $x$ in a fixed compact set, $\xi(x,\tau_1,\tau_2)\ll e^{\tau_2}$.
With no loss of generality we can assume that the slope $z$ of $\vu$
lies in the interval $[0,1)$; indeed, for any $\gamma \in \Gamma$,
$\pi \circ \Psi(\gamma \vv)$ and $\pi \circ \Psi(\vv)$ are asymptotic
under the flow $(a_s: s>0)$, and for any $\vv$, one of the elements
$\gamma_i \vv$ has slope in $[0,1)$, where  
 $$\gamma_1=e, \, \gamma_2=\left[ \begin{array}{cc} 0 &-1 \\ 1 & 0 \end{array} \right]  , \,
 \gamma_3=\left[ \begin{array}{cc} 1 &-1 \\ 1 & 0 \end{array} \right], \,  
\gamma_4=\left[ \begin{array}{cc} 1 & 1 \\ 0 & 1 \end{array} \right].$$

 Consider the half-space model of the hyperbolic space $\{ z \in {\bf
C} : \Im(z)>0\}$. Recall that  
 $$\mathcal{F}=\{ z \in {\bf C} \, : \, |z|\geq 1, \Re(z) \in [-1/2,1/2] \}$$
is a fundamental domain for the action of $\PSL(2,\ZZ)$. 
 The basepoints of $(\Psi(\vv)a_t)_{t\geq 0}$ lie at a uniformly
bounded distance from the geodesic ray $z_t=(z+i e^{-t})_{t\geq
0}$. Let $t \in [s_1,s_2]$, and $\gamma_t \in \PSL(2,\ZZ)$ such
that $\gamma_t  z_t \in \mathcal{F}$. Then the difference $|\dist(\pi
z_t, \pi \circ \Psi(\vu_0)) - \log \Im (\gamma_t z_t) |$ is bounded, so  
 $$\xi(\vv,s_1,s_2) \ll \sup_{t\in [s_1,s_2]} \Im (\gamma_t z_t).$$
 
 Consider a fixed $t>0$, and define $p,q \in \ZZ$ (depending on $t$) by
 $$\gamma_t=\left[ \begin{array}{cc} * & * \\ q & -p \end{array} \right].$$
A standard computation gives that for any $s\in \RR$,
 $$\Im (\gamma_t z_s)= \frac{1}{(qz-p)^2e^s + q^2e^{-s}}.$$
This implies that if $z \neq p/q$, the maximum of $s \mapsto \Im
(\gamma_t z_s)$ is attained for $s=-\log {|z-p/q|}$, and its value is
equal to $\frac1{2q^2|z-p/q|}$. If  $\Im (\gamma_t z_t)\geq 1$, we
have $q^2|z-p/q|\leq 1/2$ and by \cite[Theorem 184]{HW}, $p/q$ are
necessarily convergents of the continued fraction, so there exists a
$k\geq 0$ such that $|p_k|=|p|$ and $|q_k|=|q|$, and $k$ satisfies 
$ t_{k-1} \leq t \leq t_{k+1} $. 

This completes the proof in case $z$ is irrational. 
The case of rational $z$ is similar and we omit it. 
\end{proof}
The following well-known result gives a bound on the continued fraction
expansion of $\beta$-diophantine vectors. We provide a
proof for the sake of completeness.

\begin{lem} 
\name{lem: betadiophantine}
Assume $z \in [0,1)$ is $\beta$-diophantine. Then
$$\hat{\xi}(z,\tau_1,\tau_2) \ll 
e^{\frac{\beta-2}{\beta}\tau_2}.$$
\end{lem}
\begin{proof}
 Let $k$ be such that $t_{k-1}\leq \tau_2$. 
By \equ{eq: tk bounds},
$a_kq_k^2\leq e^{\tau_2}$. On the other hand, 
the assumption that $z$
is $\beta$-diophantine means that  
 $e^{t_k} \leq q_k^{\beta}/c $, so using \equ{eq: tk bounds} again we have
$a_k  \leq q_k^{\beta-2}/c.$  
Thus 
\eq{eq: min}{a_k \leq \min\left(\frac{e^{\tau_2}}{q_k^2},\
\frac{q_k^{\beta-2}}{c} \right).}
The maximum of $q\mapsto
\min(e^{\tau_2}/q^2,q^{\beta-2}/c)$ is attained when
$q=c^{1/\beta}e^{\tau_2/\beta}$, and plugging this value into
\equ{eq: min} proves the claim. 
\end{proof}
\section{Proof of Theorem  \ref{thm: main, SL(2,Z)}} 
We retain the notations of the previous section; 
for the reader's amusement we prove this time
the lower bound. Let $\delta_{\Gamma}$, and let
$\delta_0=\delta_\Gamma/24$. Define $D,f_\ell, r_\ell, R_\ell,c$ as
before. 
Write 
$$\xi_{f,T,\vu}= \xi\left( \pi \circ \Psi \left( \frac{\vu}{|\vu|} \right),\log
\left( \frac{T |\vu|}{R(f)} \right) 
,\log \left(\frac{T |\vu|}{r(f)}\right) \right).$$

Assume 
\eq{eq: T is large}{T\geq c\, D}
and 
\eq{eq: T xi large}{
T \geq \frac{\xi_{f,T,\vu} |\vu|}{R(f)}.
}

 Let $\til u = \Psi(\vu)$ and $x_0 = \pi(\til
u)$. Set 
\eq{eq: definition of alpha}{
\alpha=\left( \frac{R(f) T}{|\vu|\, \xi_{f,T,\vu} }  \right)^{\theta
\delta_0}\stackrel{\equ{eq: T xi large}} {\geq} 1.}

\begin{lem}
 For any $\ell$ such that $f_\ell$ is nonzero, we have 
$$
\xi \left( x_0a_{2\ell/\alpha},
\log \left( \frac{T-D}{R_\ell}-1 \right) \right) 
\ll
\xi_{f,T,\vu}. 
$$
\end{lem}
\begin{proof}
By \equ{eq: choice of Ri} and \equ{eq: T is large} one has 
\eq{eq: timelength1}{
\frac{T}{|\vu|}e^{-\ell/\alpha} \asymp
\frac{T-D}{R_\ell}-1. 
}
Using \equ{equ2}, one has the equality
 $$x_0a_{2\ell/\alpha}a_{ \log(T
e^{-\ell/\alpha}/|\vu|)}=\Psi\left(\frac{\vu}{|\vu|}\right)a_{\log
T+\ell/\alpha+\log |\vu|},$$ 
  and the use of inequality \equ{eq: ellbounds} proves the claim.
\end{proof}

Define $\til f_{\ell}$ and $\bar f_{\ell}$ by \equ{eq: defn til f}, \equ{eq:
defn bar f}, and set 
$$\eta = 1-\theta\delta_\Gamma/8$$
so that
\eq{eq: property of eta}{
\alpha^2 T^{\eta} =
\left( \frac{R(f) }{|\vu|\, \xi_{f,T,\vu} } \right)^{2\theta \delta_0}
T^{1-\theta \delta_0}, \ \ \  \ \frac{T}{\alpha} = \left( \frac{|\vu|\,
\xi_{f,T,\vu} }{R(f) } \right)^{\theta \delta_0} T^{1-\theta \delta_0}.  
}
Then for a lower bound we have:
\begin{align*}
S_{f,\vu}(T) & = \sum_{\gamma \in \Gamma_T} f(\gamma \vu) = \sum_{\ell}
\sum_{\gamma \in \Gamma_T} f_{\ell}(e^{\ell/\alpha} \gamma \vu) 
\stackrel{\equ{equ2}}{=} \sum_{\ell} \sum_{\gamma \in
\Gamma_T} \int_{-\infty}^{\infty}  \til f_{\ell}
(\gamma \til \vu a_{2\ell/\alpha}h_s)\, ds \\
& \stackrel{f_{\ell} \geq 0}{\geq} \sum_{\ell} \sum_{\gamma \in
\Gamma_T} \int_{-\left((T-D)/R_{\ell}- 
1\right)}^{(T-D)/R_{\ell}-1}  \til f_{\ell}
(\gamma \til \vu a_{2\ell/\alpha}h_s)\, ds \\ 
& \stackrel{\equ{eq: for boundary effect} }{\geq} 
\sum_{\ell}  \int_{-\left((T-D)/R_{\ell}-
1\right)}^{(T-D)/R_{\ell}-1} \bar{f}_{\ell} (x_0 a_{2\ell/\alpha} h_s)\, ds \\
& \stackrel{\equ{eq: Strombergsson}, \equ{eq: timelength1} }{\geq} \frac{2}{\mu(\gmg)} \sum_{\ell}
\left(\frac{T-D}{R_{\ell}} - 1 \right) 
\int_{\gmg} \bar{f}_{\ell} \, d\mu \\
& - c_1\sum_{\ell} \|\bar{f}_{\ell}\|_{\theta}
\left(\frac{T-D}{R_{\ell}}-1\right)^{\eta} \xi_{f,T,\vu}^{1-\eta} \\
& \stackrel{\equ{eq: normalization again}, \equ{eq:
holdercontrol}
}{\geq} \frac{2}{\mu(\gmg)}
\sum_{\ell} \frac{T-D-R_\ell}{R_{\ell}} \int_{\PP} f_{\ell} \, dx \\
& - c_2 T^{\eta}
\sum_{\ell} \|f_{\ell}\|_{\theta} R_\ell^{-\eta}
\left(1- \frac{D}{T} -\frac{R_\ell}{T}\right)^{\eta} \xi_{f,T,\vu}^{1-\eta}
\\
& \stackrel{\equ{eq: estimateofholder},
\equ{eq: intup}, \equ{eq: boundonrl}
}{\geq} \frac{2(T - c_3D) }{\mu(\gmg)}  
e^{-2/\alpha} \int_{\PP} \frac{f(\vv)}{\vv \star \vu} \, d\vv \\
& \; \; \; \; \ \ 
-  c_4 \alpha T^{\eta} \|f\|_{\theta} \left( \frac{R(f)}{|\vu|}\right)^{\eta}  
\left (\alpha \log v^{(\vu)}(f)+2 \right) \xi_{f,T,\vu}^{1-\eta} \\ 
& \stackrel{e^{2/\alpha}=1+O(1/\alpha) }{\geq}
\frac{2T}{\mu(\gmg)} \int_{\PP} 
\frac{f(\vv)}{\vv \star \vu} \, d\vv - c_5 \left( D+\frac{T}{\alpha}\right) \int_{\PP} 
\frac{f(\vv)}{\vv \star \vu} \, d\vv \\
& \; \; \; \; \ \
- c_6 \alpha^{2} T^{\eta}\ \left(\log v^{(\vu)}(f)  +1\right)
\|f\|_{\theta} \left( \frac{R(f)}{|\vu|}\right)^{\eta} \xi_{f,T,\vu}^{1-\eta}.
\end{align*}
Using \equ{eq: intnosolarge} and \equ{eq: property of eta} we obtain
$$ 
\frac{2T}{\mu(\gmg)} \int_{\PP} \frac{f(\vv)}{\vv \star \vu} \, d\vv
- S_{f,\vu}(T)  
\ll  \|f \|_{\theta}\frac{R(f)}{|\vu|} \left( D+B T^{1-\theta
\delta_0} \xi_{f,T,\vu}^{\theta \delta_0} \right), 
$$
as claimed. The proof of the opposite bound is similar.
\qed

\begin{proof}[Proof of Corollary \ref{cor: main, SL(2,Z)}]
As in the proof of Corollary \ref{cor: applications}, we apply Theorem
\ref{thm: main, SL(2,Z)} to $f$ and  
$$\vw= \vw(T)= \frac{\vu}{T^{\alpha}}.$$
Since the slope of $\vu$ was assumed be $\beta$-diophantine, 
$$\hat{\xi}\left(\vw,\log \left( \frac{T |\vw|}{R(f)} \right),\log
\left(\frac{T |\vw|}{r(f)}\right) \right)  \ll \left( \frac{T
|\vw|}{r(f)} \right)^{\frac{\beta-2}{\beta}}  \ll
T^{\frac{(1-\alpha)(\beta-2)}{\beta}}.$$ 
In order to apply Theorem \ref{thm: main, SL(2,Z)}, we need to check that
$$D(\vw(T), f) \leq c T^{|\alpha|} \ll T,$$
which is always true, and that 
$$\frac{|\vw|}{R(f)} \hat{\xi}\left(\vw,\log \left( \frac{T
|\vw|}{R(f)} \right),\log \left(\frac{T |\vw|}{r(f)}\right) \right)
\ll  T^{-\alpha}T^{ \frac{(1-\alpha)(\beta-2)}{\beta} } \ll T,$$ 
which is true for all large enough $T$ by virtue of our assumption that 
$\alpha > - \frac{1}{\beta-1}.$
Therefore for some $C$ depending on $f$ and $\vu$ but independent of $T$,
\[ 
\left|S_{f, \vw}(T) -
\frac{2T^{1+\alpha}}{\mu(\gmg)} \int_{\PP} 
\frac{f(\vv)}{\vv \star \vu}\, d\vv \right|
<CT^\alpha\left(T^{|\alpha|}+T^{1-\theta
\delta_0(1+\alpha)}T^{\frac{\theta\delta_0(1-\alpha)(\beta-2)}{\beta}} \right).
\]
Dividing through by $T^{1+\alpha}$ gives
$$\left| \frac{1}{T^{1+\alpha}} S_{f, \vw}(T) -
\frac{2}{\mu(\gmg)} \int_{\PP} 
\frac{f(\vv)}{\vv \star \vu}\, d\vv \right| \leq C\left(
T^{|\alpha|-1}+T^{-2\theta\delta_0
\left(\frac{\alpha(\beta-1)+1}{\beta} \right)}\right).$$ 
Taking $\delta=\min \big( 1-|\alpha|, 2\theta\delta_0
\left( \alpha + (1-\alpha)/\beta \right)
\big ) >0$, we obtain \equ{main cor}. 
\end{proof}

\begin{proof}[Proof of Corollary \ref{cor : Dirichlet}]
 Let $f$ be a nonnegative smooth function, vanishing outside a disk of
radius $1$. Then for $\lambda>0$, the function 
 $f_\lambda(\vw)=f(\lambda^{-1}(\vw-\vv))$ vanishes outside a disk of
radius $\lambda$ centered on $\vv$, and is Lipschitz. 
 For all $\lambda$ small enough, 
 $$||f_\lambda ||_{1} \ll \lambda^{-1}, \   \  \int_\PP
\frac{f_\lambda(\vw)}{\vw \star \vu} d\vw \asymp \lambda^2, \ \ 
R(f_\lambda) \ll 1,  \  \mathrm{and \ } \ \  \hat{\xi}_{f_\lambda,T,\vu}
\ll T^{\frac{\beta-2}{\beta}}$$
(by Lemma \ref{lem: betadiophantine}). 
Taking $\theta =1$ in \equ{eq: mainfv2}, we find that there are positive
constants $c_i$ such that
$$S_{f_\lambda, \vu}(T) \geq  c_1 T\lambda^{2} -
c_2\lambda^{-1}\left(c_3+c_4T^{1- 2\delta_0 / \beta} \right).
$$
Therefore there is a positive constant $C$ such that if we set 
$\lambda=CT^{-\frac{2\delta_0}{3\beta}}$, then $S_{f_\lambda, \vu}(T) >
0$ for all large enough $T$. This proves (1). 

The idea for (2) is the following classical geometrical
property: there exists a fixed compact subset in $\gmg$ which every
nondivergent geodesic intersects infinitely many times. So for all
$\vu$ with irrational slope, there 
exists a sequence $(s_i)_i$ tending to infinity, such that
$\xi(\Psi(\vu/|\vu|),s_i)$ is uniformly bounded. Since we have
$R^{(\vu)}(f_\lambda) \asymp r^{(\vu)}(f_\lambda) \asymp 
\vv \star \vu$ for $\lambda<1$, we find that
$\xi_{f_\lambda,T_i,\vu}$ is bounded for the times 
$$T_i= \frac{\vv\star \vu}{|\vu|^2}e^{s_i}.$$
We now proceed as before,
but with $\xi_{f_\lambda,T_i,\vu}$  instead of 
$\hat{\xi}_{f_\lambda,T,\vu}$ (which is legitimate, because in the
proof of Theorem \ref{thm: main, SL(2,Z)}, we used $\xi$ instead of
$\hat{\xi}$) . 
\end{proof}

\section{Acknowledgments}
We thank Alex Gorodnik, S\'ebastien Gouez\"el and Amos Nevo for useful discussions. The first
author would like to thank the Center for Advanced Studies in
Mathematics at Ben Gurion University for 
its hospitality during a stay where this work was first conceived. The
work of the second author was supported by the Israel Science
Foundation.

\end{document}